\documentclass{amsart}

\usepackage{amsmath}
\usepackage{amsthm}
\usepackage{hyperref}
\usepackage{amsfonts,graphics,amsthm,amsfonts,amscd,latexsym}
\usepackage{epsfig}
\usepackage{flafter}
\usepackage{mathtools}
\usepackage{comment}
\usepackage{stmaryrd}
\usepackage{enumitem}

\usepackage{mathabx,epsfig}

\hypersetup{
    colorlinks=true,    
    linkcolor=blue,          
    citecolor=blue,      
    filecolor=blue,      
    urlcolor=blue           
}
\usepackage{tikz}
\usetikzlibrary{graphs,positioning,arrows,shapes.misc,decorations.pathmorphing}

\tikzset{
    >=stealth,
    every picture/.style={thick},
    graphs/every graph/.style={empty nodes},
}

\tikzstyle{vertex}=[
    draw,
    circle,
    fill=black,
    inner sep=1pt,
    minimum width=5pt,
]
\usepackage[position=top]{subfig}
\usepackage{amssymb}
\usepackage{color}

\calclayout

\usetikzlibrary{decorations.pathmorphing}
\tikzstyle{printersafe}=[decoration={snake,amplitude=0pt}]

\newcommand{\yy}{\mathcal{Y}}
\newcommand{\cc}{\mathcal{C}}
\newcommand{\pp}{\mathbb{P}}

\newcommand{\qq}{\mathbb{Q}}

\newcommand{\ee}{\mathcal{E}}

\newcommand{\TT}{\mathbb{T}}
\newcommand{\oo}{\mathcal{O}}

\def\O#1.{\mathcal {O}_{#1}}			
\def\pr #1.{\mathbb P^{#1}}				
\def\af #1.{\mathbb A^{#1}}			
\def\ses#1.#2.#3.{0\to #1\to #2\to #3 \to 0}	
\def\xrar#1.{\xrightarrow{#1}}			
\def\K#1.{K_{#1}}						
\def\bA#1.{\mathbf{A}_{#1}}			
\def\bM#1.{\mathbf{M}_{#1}}				
\def\bL#1.{\mathbf{L}_{#1}}				
\def\bB#1.{\mathbf{B}_{#1}}				
\def\bK#1.{\mathbf{K}_{#1}}			
\def\subs#1.{_{#1}}					
\def\sups#1.{^{#1}}

\usepackage{tikz}
\usetikzlibrary{matrix,arrows,decorations.pathmorphing}

\newtheorem{theorem}{Theorem}[section]
\newtheorem{lemma}[theorem]{Lemma}
\newtheorem{proposition}[theorem]{Proposition}
\newtheorem{corollary}[theorem]{Corollary}

\newtheorem{conjecture}[theorem]{Conjecture}

\newtheorem{notation}[theorem]{Notation}
\newtheorem{definition}[theorem]{Definition}
\newtheorem{example}[theorem]{Example}

\newtheorem{question}[theorem]{Question}
\newtheorem{construction}[theorem]{Construction}

\newtheorem{remark}[theorem]{Remark}

\theoremstyle{remark}

\numberwithin{equation}{section}

\usepackage[all]{xy}

\begin{document}

\title[A hyperbolicity conjecture for adjoint bundles]{A hyperbolicity conjecture for adjoint bundles}

\author[J.~Moraga]{Joaqu\'in Moraga}
\address{UCLA Mathematics Department, Box 951555, Los Angeles, CA 90095-1555, USA
}
\email{jmoraga@math.ucla.edu}

\author[W. Yeong]{Wern Yeong}
\address{UCLA Mathematics Department, Box 951555, Los Angeles, CA 90095-1555, USA
}
\email{wyyeong@math.ucla.edu}


\begin{abstract}
Let $X$ be a $n$-dimensional smooth projective variety and $L$ be an ample Cartier divisor on $X$. We conjecture that a very general element of the linear system $|K_X+(3n+1)L|$
is a hyperbolic algebraic variety. 
This conjecture holds for some classical varieties: surfaces, products of projective spaces, and Grassmannians. In this article, we investigate the conjecture for $X$ a toric variety.
We confirm the conjecture in the case of smooth projective toric varieties. 
When $X$ is a Gorenstein toric variety, we show that 
$|K_X+(3n+1)L|$ is pseudo hyperbolic.
For a Gorenstein toric threefold $X$, we show that 
$|K_X+9L|$ is hyperbolic.
\end{abstract}

\maketitle

\setcounter{tocdepth}{1} 
\tableofcontents

\section{Introduction}

Given a complex space $X$, Kobayashi
constructed an invariant pseudo-distance 
by means of chains of holomorphic discs (see, e.g.,~\cite{Kob70}).
This invariant pseudo-distance is a natural extension of the classic Poincar\'e distance
on the complex unit disc.
A complex space $X$ is said to be {\em Kobayashi hyperbolic} if Kobayashi's pseudo-distance is indeed a distance function.
In~\cite{Dem97}, Demailly introduced the notion of {\em algebraic hyperbolicity} for a smooth projective variety $X$ as an algebraic analog of Kobayashi hyperbolicity (see Definition~\ref{def:hyperb}). 
The algebraic hyperbolicity is a uniform lower bound of the genus of curves $C$ in $X$ that is linear in terms of the degree once a polarization on $X$ is fixed.
It is expected that an algebraically hyperbolic 
smooth projective variety is Kobayashi hyperbolic (see, e.g.~\cite{JK20}).
In any case, hyperbolic varieties should contain no smooth rational curves or elliptic curves.
It is a general philosophy that generic hypersurfaces of large degree on smooth projective varieties are 
hyperbolic. 
For instance, for every ample Cartier divisor $A$ on a smooth projective variety $X$, there exists a constant $d_0:=d_0(X,A)$ such that a general hypersurface $H\in |A^{d}|$ is Kobayashi hyperbolic provided that $d\geq d_0$~\cite{Bro17}.
Moreover, for $A$ a very ample Cartier divisor on $X$, the previous statement works for $d_0\leq (n+1)^{(2n+6)}$ where $n$ is the dimension of $X$~\cite{Den16}.
In the case of algebraic hyperbolicity,
Clemens~\cite{Cle86} and Ein~\cite{Ein88} proved that a very general hypersurface in $\pp^n$ of degree at least $2n$ is algebraically hyperbolic.
This bound was improved by Voisin~\cite{Voi96,Voi98} to $2n-1$ for $n\geq 4$ and by Pacienza~\cite{Pac04}, Clemens and Ran~\cite{CR04}, and the second author~\cite{Yeo24} to $2n-2$ for $n\geq 5$.
In this article, we pose the following conjecture about the algebraic hyperbolicity of adjoint linear systems on smooth projective varieties.

\begin{conjecture}
\label{conj:main}
{\em
Let $L$ be an ample Cartier divisor on a smooth projective variety $X$ of dimension $n\geq 2$. Then the linear system $|K_X+(3n+1)L|$ is hyperbolic.\footnote{This means that a {\em very general} element is algebraically hyperbolic.}
}
\end{conjecture}

Conjecture~\ref{conj:main} is motivated by the aforementioned results on algebraic hyperbolicity of hypersurfaces, Fujita's conjecture~\cite{Fuj87}, and the syzygies of adjoint linear systems~\cite{BL23}.
The bound $3n+1$ in the above conjecture is sharp
by considering curves in $\pp^2$.
However, by considering very general hypersurfaces in $\pp^n$, we expect the sharp bound to be lower in higher dimensions; namely $3n$ when $n=3$ and $3n-1$ when $n\geq5$.
Conjecture~\ref{conj:main} is known in the cases when $X$ is $\pp^n$, a product of projective spaces,
the Grassmannian $G(k,n)$, or a product of Grassmannians (see Example~\ref{ex:prod-project} and Example~\ref{ex:prod-grass}).
A weaker version of Conjecture~\ref{conj:main}, namely the pseudo hyperbolicity of the linear system, is known for abelian varieties~\cite{Kaw80}.
The main aim of this article is to provide new evidence for Conjecture~\ref{conj:main}.
The hyperbolicity of hypersurfaces in toric varieties has been studied by many authors~\cite{Ikeda09,HaaseIlten21,CoskunRiedl23,Rob21}. 
For this reason, we focus on Conjecture~\ref{conj:main} in the case that $X$ is a toric variety.
It is worth mentioning that Conjecture~\ref{conj:main} is a natural extension of Fujita's conjecture on the base point freeness and very ampleness of adjoint linear systems.
Fujita's conjecture is known to hold for mildly singular toric varieties.
Our first result is a positive answer to Conjecture~\ref{conj:main} for smooth toric varieties.

\begin{theorem}\label{thm:Fujita-hyperbolicity-toric}
{\em
Let $X$ be a Gorenstein projective toric variety of dimension $n$, and $L$ an ample Cartier divisor on $X$.
Then $|K_X+(3n+1)L|$ is pseudo hyperbolic. If $X$ is smooth, then $|K_X+(3n+1)L|$ is hyperbolic.
}

\end{theorem} 

We say that a linear system is \emph{hyperbolic} if a very general element of the linear system is an algebraically hyperbolic variety.
We say that it is \emph{pseudo hyperbolic} if a very general element of the linear system is algebraically hyperbolic up to a proper subvariety
(see Definition~\ref{def:hyperb}).
The second result of this article is a hyperbolicity result for adjoint linear systems
on Gorenstein toric $3$-folds.

\begin{theorem}\label{thm:Fujita-hyperbolicity-toric-3-folds}
Let $X$ be a Gorenstein toric $3$-fold.
Let $L$ be an ample Cartier divisor on $X$.
Then the linear system $|K_X+9L|$ is hyperbolic, i.e., a very general element of this linear system is an algebraically hyperbolic surface.
\end{theorem} 

The statement of the previous theorem is sharp.
Indeed, a very general quintic in $\pp^3$ is algebraically hyperbolic, while a very general quartic in $\pp^3$ is not (see, e.g.,~\cite{CR19}). 
We note that our proof applies to all Gorenstein toric threefolds except for $\pp^3$, which is addressed in~\cite{CR19}.
Theorem~\ref{thm:Fujita-hyperbolicity-toric-3-folds} produces many cases of $3$-dimensional singular varieties in which Conjecture~\ref{conj:main} holds. 
Furthermore, even in the case of Fano varieties the previous theorem gives plenty of examples.
Indeed, there are 4319 isomorphism classes of Gorenstein toric Fano $3$-folds~\cite{KS98}.

Both Theorem~\ref{thm:Fujita-hyperbolicity-toric} (the smooth case) and Theorem~\ref{thm:Fujita-hyperbolicity-toric-3-folds} are proved by induction on the dimension.
We briefly explain the ideas that lead to the proof of the smooth case of Theorem~\ref{thm:Fujita-hyperbolicity-toric}.
First, we show that the linear system $|K_X+(3n+1)L|$ is pseudo hyperbolic.
More precisely, in Theorem~\ref{thm:sec-3-gen-hyp-smooth}, we show that $|K_X+(3n+1)L|$ is pseudo hyperbolic up to $D_1\cup\dots \cup D_k$ where the $D_i$'s are the prime torus invariant divisors of $X$ for which $L+D_i$ is not nef.
Then, we inductively show that the restriction of the linear system $|K_X+(3n+1)L|$
to each prime torus invariant divisor $D$ is hyperbolic.
To do so, we will use Fujita's conjecture for toric varieties~\cite{Payne06} and write the adjoint linear system as a sum of a nef divisor and several ample divisors.
Moreover, we will show that the map 
\begin{equation}\label{eq:surj} 
H^0(X,\mathcal{O}_X(K_X+(3n+1)L))
\rightarrow 
H^0(D,\mathcal{O}_{D}((K_X|_{D}+(3n+1)L|_{D})))
\end{equation}
is surjective as an application of Musta\c{t}\u{a}'s vanishing~\cite{Mustata02}. Therefore, a very general element of the former linear system restricts 
to a very general element of the latter linear system.
The surjectivity of~\eqref{eq:surj} and hyperbolicity of the restriction of $|K_X+(3n+1)L|$ to each $D$ give the hyperbolicity of $|K_X+(3n+1)L|$.

In Section~\ref{sec:examples}, we show that 
there are triples $(X,L,D)$ consisting of: 
\begin{itemize}
\item a Gorenstein toric Fano $3$-fold $X$, 
\item an ample Cartier divisor $L$ on $X$, and
\item  a prime torus invariant divisor $D$ on $X$
\end{itemize} 
such that the divisor $L+D$ is not nef. 
Indeed, in Example~\ref{ex:xld}, we classify such cases when $X$ is a smooth toric Fano $3$-fold. 
This shows that Theorem~\ref{thm:sec-3-gen-hyp-smooth} does not immediately imply the smooth case of Theorem~\ref{thm:Fujita-hyperbolicity-toric-3-folds} and an analysis of the finitely many curves contained in the divisors $D_i$ is needed.

\subsection*{Acknowledgements}

The authors would like to thank Damian Brotbek, Alex K\"uronya, Antonio Laface, Robert Lazarsfeld, and Pedro Montero for very useful comments. We are grateful to Haesong Seo for communicating a gap in an earlier draft.

\section{Preliminaries}
\label{sec:prelim}

We will work over the complex numbers throughout. In this section, we recall some basics related to the hyperbolicity of hypersurfaces and syzygy bundles.
For the basics about toric varieties, we refer the reader to~\cite{CLS11}.

\begin{definition}
\label{def:hyperb}
{ \em
Let $E$ be a nef Cartier divisor on a projective variety $X$.
We say that the linear system $\left|E\right|$ is \textit{hyperbolic} if the variety $D_E$ in $X$ defined by the vanishing of a very general element of $\left|E\right|$ is Demailly algebraically hyperbolic, namely,
there is a constant $\varepsilon > 0$ and an ample divisor $A$ on $D_E$ such that every nonconstant map $f: C \to D_E$ from a smooth projective curve $C$ satisfies
\begin{equation}
\label{eqn:hyperbolicity-definition}
2g(C)-2\geq \varepsilon \deg f^*A,
\end{equation}
where $g(C)$ denotes the geometric genus of $C$.

We say that $E$ is \textit{pseudo hyperbolic up to a proper subvariety $Z\subsetneq D_E$} if $D_E$ is Demailly algebraically hyperbolic outside of $Z$, namely there is a constant $\varepsilon > 0$ and an ample divisor $A$ on $D_E$ such that every map $f: C \to D_E$ as above where $f(C)\not\subset Z$ satisfies \eqref{eqn:hyperbolicity-definition}.
}
\end{definition}

We use the following notation for the rest of this section, unless otherwise noted.

\begin{notation} 
\label{notation}
{\em 
In Section~\ref{sec:prelim}, $X$ denotes a smooth projective toric variety with torus $\TT$ and torus-invariant prime divisors $D_1,\ldots,D_s.$ 
Suppose that $\ee$ is a big and nef line bundle on $X$ whose general section defines a smooth, irreducible subvariety $Y$ of $X$.
We may assume without loss of generality that $\ee$ is $\TT$-equivariant.
Let $C$ be a curve contained in $Y$ 
with geometric genus $g$ and degree $e$ with respect to an ample line bundle $\mathcal{A}$ on $X$.
}
\end{notation} 

Let $C$ and $Y$ be as in the notation above. We recall a construction to {\em spread out} $C$ and $Y$ into a versal family (see e.g. \cite{Ein88, ClemensRan04, CoskunRiedl23}).

\begin{construction}
\label{construction}
{\em
Let $B_1=H^0(X,\ee)$, and let $\yy_1\rightarrow B_1$ be the universal hypersurface given by the vanishing locus of a section of $\ee$.
Let $\psi:\mathcal{H}\rightarrow B_1$ be the relative Hilbert scheme parametrizing curves of geometric genus $g$ and degree $e$ with respect to $\mathcal{A}$.
Since we assume that the subvariety $Y\subset X$ defined by a general section of $\ee$ contains such a curve $C$, the morphism $\psi$ is a dominant map.

Let $\yy_2 := \yy_1\times_{B_1} \mathcal{H}$, and let $\cc_1\rightarrow \mathcal{H}$ denote the universal curve, so that we have 
\[
\xymatrix{
\cc_1\, \ar@{^{(}->}[rr] \ar[dr] && \yy_2 \ar[dl]\\
 &\mathcal{H}&
}
\]
By a standard argument (see \S1 in \cite{ClemensRan04}), we may replace $\mathcal{H}$ with a $\TT$-invariant subvariety $\mathcal{H}_1 \subset \mathcal{H}$ such that $\psi\vert_{\mathcal{H}_1}:\mathcal{H}_1\rightarrow B_1$ is \'etale. We explain the details of this in the following paragraph.

As $\ee$ is big, the stabilizer of $\mathbb{T}$ at a general point of $B_1$ is finite. Therefore, at a general point $b\in B_1$, we may choose a subvariety $B_0\subseteq B_1$ through $b$ such that:
\begin{enumerate}
\item $B_0$ is of complementary dimension to the $\TT$-orbit $\TT\cdot b$ of $b$, which has dimension equal to $\dim \TT = n$ (i.e. $\dim B_0 = \dim B_1-n$), and
\item $B_0$ meets the orbit $\TT\cdot b$ transversally at $b$.
\end{enumerate}
Then, at a general point $h\in \psi^{-1}(B_0),$ we may choose a subvariety $\mathcal{H}_0\subseteq \psi^{-1}(B_0)$ through $h$ such that $\psi$ is \'etale over $B_0$.
Now let $\mathcal{H}_1= \TT\cdot \mathcal{H}_0\subseteq \mathcal{H}$, so we have a $\TT$-invariant subvariety $\mathcal{H}_1\subseteq \mathcal{H}$ such that $\psi\vert_{\mathcal{H}_1}:\mathcal{H}_1\rightarrow B_1$ is \'etale.

Let us denote the pullback of the above diagram over $\mathcal{H}_1$ by
\[
\xymatrix{
\cc_2\, \ar@{^{(}->}[rr] \ar[dr] && \yy_3 \ar[dl]\\
 &\mathcal{H}_1&
}
\]
After taking a $\TT$-equivariant resolution of $\cc_2 \to \mathcal{H}_1$ and possibly restricting $\mathcal{H}_1$ to some $\TT$-invariant open subset $B\subset \mathcal{H}_1$, we obtain a smooth family $\cc \rightarrow B$ whose fibers are smooth curves of genus $g$.
Pulling back $\yy_3\rightarrow \mathcal{H}_1$ to $\yy \rightarrow B$, we obtain the following diagram
\[
\xymatrix{
\cc\, \ar[rr]^-{g} \ar[dr] && \yy \ar[dl]\\
 &B &
}
\]
where $g\colon \cc \to \yy$ is a generically injective map.
We denote by $\pi_1:\yy \rightarrow B$ and $\pi_2: \yy \rightarrow X$ the natural projection maps.
}
\end{construction}

Since $\ee$ is globally generated, the natural map $T_\yy \rightarrow \pi_2^*T_X$ is surjective, and we define the vertical tangent sheaf $T_{\yy/X}$ as the kernel sheaf in the short exact sequence
\begin{equation}
\label{eqn:vert-tgt-sheaf}
0 \rightarrow T_{\yy/X} \rightarrow T_\yy \rightarrow \pi_2^*T_X \rightarrow 0.
\end{equation}
We may also define the vertical tangent sheaf $T_{\cc/X}$ as the kernel sheaf in the exact sequence
\begin{equation}
\label{eqn:vert-tgt-sheaf-C}
0 \rightarrow T_{\cc/X} \rightarrow T_\cc \rightarrow g^*\pi_2^*T_X.
\end{equation}
We have the following standard short exact sequence on $\cc$:
\begin{equation}
\label{eqn:standard-ses-C}
0 \rightarrow T_\cc \rightarrow g^*T_\yy \rightarrow N_{g/\yy} \rightarrow 0.
\end{equation}
One sees that $T_{\cc/X}$ injects naturally into $g^*T_{\yy/X}$, and we denote by $\mathcal{K}$ the quotient sheaf in the short exact sequence
\begin{equation}
\label{eqn:vertical-ses-C}
0 \rightarrow T_{\cc/X} \rightarrow g^*T_{\yy/X} \rightarrow \mathcal{K} \rightarrow 0.
\end{equation}

\begin{remark}
\label{rem:vert-tgt-sheaf-C}
{\em
If we assume further that a general section $Y$ of $\ee$ contains such a curve $C$ as in Notation~\ref{notation} that meets the torus $\TT$, then since $\cc$ is constructed to be $\TT$-invariant, we have that $\pi_2\circ g\colon \cc \to X$ dominates $\TT$, and that $T_\cc \to g^*\pi_2^*T_X$ from \eqref{eqn:vert-tgt-sheaf-C} is generically surjective over $(\pi_2\circ g)^{-1}(\TT).$
}
\end{remark}

\begin{definition}
\label{def:lazarsfeld-bundle}
{ \em 
Let $X$ be a projective variety.
The \textit{Lazarsfeld kernel bundle} or \textit{syzygy bundle} $M_\ee$ with respect to a globally generated vector bundle $\ee$ on $X$ is the kernel bundle in the following short exact sequence:
\begin{equation*}
0 \rightarrow M_\ee \rightarrow H^0(X,\ee) \otimes \oo_X\xrightarrow{\text{ev}} \ee \rightarrow 0.
\end{equation*}
}
\end{definition}

\begin{lemma}[cf. Proposition 2.1 in \cite{CoskunRiedl23}]
\label{lem:coskun-riedl-prop2.1}
{\em
Let $X$ be a smooth projective toric variety with torus $\TT$, and $\mathcal{A}$ an ample line bundle on $X$.
Suppose that a general hypersurface $Y$ in the linear system of a globally generated, $\TT$-equivariant line bundle $\ee$ on $X$ contains a curve $C$ of geometric genus $g$ and degree $e$ with respect to $\mathcal{A}$.
Take $\cc,\yy,B$ as in Construction~\ref{construction}. Then:
\begin{enumerate}
\item For a general $b\in B$, $N_{g/\yy}\vert_{C_b} \simeq N_{g_b/Y_b}$ where $g_b \colon C_b \to Y_b$ is the restriction of $g$ over $b$.
\item $T_{\yy/X}\simeq \pi_2^*M_\ee.$
\item If we assume further that a general $Y\in \left|\ee\right|$ contains such a curve $C$ that meets the torus $\TT$, then the sheaf $\mathcal{K}$ from \eqref{eqn:vertical-ses-C} injects into $N_{g/\yy}$ with torsion cokernel.
\end{enumerate}
}
\end{lemma}

Note that Lemma~\ref{lem:coskun-riedl-prop2.1}(3) follows from Remark~\ref{rem:vert-tgt-sheaf-C}.

\begin{lemma}
\label{lem:mult-map-surj}
{\em
Let $X$ be a smooth projective toric variety, $N$ a nef Cartier divisor on $X$, and $L$ an ample Cartier divisor on $X$.
Then, the following multiplication maps are surjective:
\begin{enumerate}
\item $H^0(X,\oo_X(N+(2n-1)L))\otimes H^0(X,\oo_X(L))\rightarrow H^0(X,\oo_X(N+2nL)$.
\item $H^0(X,\oo_X(N+2nL))\otimes H^0(X,\oo_X(L))\rightarrow H^0(X,\oo_X(N+(2n+1)L))$.
\end{enumerate}
}
\end{lemma}

\begin{proof}
This follows from the Castelnuovo--Mumford lemma (see \cite[Lemma 1.6]{MukherjeeRaychaudhury20} and \cite[Theorem 2]{Mumford69}), since the divisors $N+(2n-1-i)L$ and $N+(2n-i)L$ are ample for $0 \leq i \leq n$.
\end{proof}

\section{Hyperbolicity in toric varieties}
\label{sec:gen-hyp}

In this section, we prove some statements regarding the algebraic hyperbolicity of very general hypersurfaces in Gorenstein projective toric varieties. 
In Subsection~\ref{subsec:gen-hyperb-toric}, we study the pseudo hyperbolicity of linear systems,
while in Subsection~\ref{subsec:hyperb-toric}, we study the hyperbolicity of these linear systems on smooth projective toric varieties by applying the result on pseudo hyperbolicity and induction on the dimension. 

\subsection{Pseudo hyperbolicity in toric varieties}
\label{subsec:gen-hyperb-toric}

In this subsection, we prove a statement about the pseudo algebraic hyperbolicity of very general hypersurfaces in Gorenstein projective toric varieties $X$.
More specifically, if $X$ has dimension $n$, and $N$ is a nef divisor and $L$ an ample divisor on $X$, then a very general hypersurface in the linear system $\left|N+2nL\right|$ is algebraically hyperbolic up to the union of all prime torus-invariant divisors of $X$.
In the smooth case, we are able to show further that the hypersurface is algebraically hyperbolic away from the union of those prime torus invariant divisors $D$ for which $L+D$ is not nef. 
In the next subsection, we will use Theorem~\ref{thm:sec-3-gen-hyp-smooth} to settle Conjecture~\ref{conj:main} for smooth projective toric varieties using an induction argument on the dimension of $X$. 
We note that the stronger control of the exceptional locus given by Theorem~\ref{thm:sec-3-gen-hyp-smooth}(2) is not required for this induction argument.

\begin{theorem}
\label{thm:sec-3-gen-hyp-gor}
{\em
Let $X$ be a Gorenstein projective toric variety of dimension $n$ with torus $\TT$, torus-invariant prime divisors $D_1,\ldots,D_s$. 
Let $N$ be a nef Cartier divisor on $X$
and $L$ be an ample Cartier divisor on $X$. 
Then $N+2nL$ is pseudo hyperbolic modulo the toric boundary $X\setminus \TT$.
}
\end{theorem}

\begin{theorem}
\label{thm:sec-3-gen-hyp-smooth}
{\em
Let $X$ be a smooth projective toric variety of dimension $n$ with torus $\TT$ and torus-invariant prime divisors $D_1,\ldots,D_s$. 
Let $N$ be a nef Cartier divisor on $X$
and $L$ be an ample Cartier divisor on $X$. 
Then, the following statements hold.
\begin{enumerate}
\item $N+2nL$ is pseudo hyperbolic modulo the toric boundary $X\setminus \TT.$
\item If $\{D_1,\ldots,D_k \}$ is the set of torus-invariant prime divisors such that $L+D_i$ is not nef, then $N+2nL$ is pseudo hyperbolic modulo $D_1\cup \cdots \cup D_k$.
In particular, if $L+D$ is nef for every torus-invariant prime divisor $D$, then $N+2nL$ is hyperbolic.
\end{enumerate}
}
\end{theorem}

The strategy for the proof of Theorems~\ref{thm:sec-3-gen-hyp-gor} and \ref{thm:sec-3-gen-hyp-smooth} is based on  Lazarsfeld kernel bundles~\cite{CoskunRiedl23} and the ideas of Ikeda relating the integer decomposition property of polytopes to the hyperbolicity of hypersurfaces in toric varieties~\cite{Ikeda09}.
The integer decomposition property of polytopes is a problem emanating from convex geometry (see, e.g.,~\cite{Oda08}).
In our setting, instead of studying the integer decomposition property, we make use of the positivity of $N+2nL$ along with results from \cite{HMP10} on the nefness and global generation of certain twists of the Lazarsfeld kernel bundle on toric varieties.

We say that a vector bundle $\mathcal{F}$ on $X$ is \emph{nef} if $\oo_{\pp\mathcal{F}}(1)$ is a nef line bundle on $\pp\mathcal{F}.$
By the Barton-Kleiman criterion (see, e.g., ~\cite[Proposition 6.1.18]{Lazarsfeld2}), $\mathcal{F}$ is nef if and only if for any finite map $\nu\colon C \to X$ from a smooth projective curve $C$ and any quotient line bundle $\mathcal{Q}$ of $\nu^*\mathcal{F}$, one has $\deg \mathcal{Q} \geq0.$
Hence, we will say that $\mathcal{F}$ is \emph{generically nef} if there is some proper subvariety $Z\subsetneq X$ such that the Barton-Kleiman criterion is satisfied for all such finite maps $\nu\colon C \to X$ with $\nu(C)\not\subseteq Z.$
Note that on a toric variety, nef line bundles are globally generated, but this does not generalize to vector bundles of higher rank \cite{HMP10}.

\begin{lemma}
\label{lem:me}
{\em
Let $X$ be a Gorenstein projective toric variety of dimension $n$.
Let $N$ be a nef Cartier divisor on $X$, $L$ an ample Cartier divisor on $X$, and $\ee:=\oo_X(N+2nL)$. 
Then $M_\ee\otimes \oo_X(L)$ is a nef vector bundle. 
If $X$ is smooth, then it is globally generated.
}
\end{lemma}

\begin{proof}
    The Gorenstein case follows immediately from Proposition~4.6 in \cite{HMP10}. 
    When $X$ is smooth, by Proposition 4.19 in \cite{HMP10}, the global generation of $M_\ee\otimes \oo_X(L)$ follows from the surjectivity of the multiplication map 
    \[H^0(X,\oo_X(N+(2n-1)L))\otimes H^0(X,\oo_X(L))\rightarrow H^0(X,\oo_X(N+2nL), \] which is proven in Lemma~\ref{lem:mult-map-surj}(1).
\end{proof}

\begin{proposition}
\label{prop:universal-tangent-sheaf-glob-gen}
{\em
Let $X$ be a smooth projective toric variety with torus $\TT$ and torus-invariant prime divisors $D_1,\ldots,D_s$.
Let $D$ be a nef Cartier divisor on $X$, and $L$ an ample Cartier divisor on $X$. 
Suppose that $Y\subset X$ is a general hypersurface in the linear system of $\ee:= \oo_X(N+2nL)$.
Then,
\begin{enumerate}
\item $T_{\yy/X}\vert_Y\otimes \oo_X(L)\vert_Y$ is globally generated.
\item  If $\{D_1,\ldots,D_k \}$ is the set of torus-invariant prime divisors such that $L+D_i$ is not nef, then \linebreak $T_{\yy}\vert_Y\otimes \oo_X(L)\vert_Y$ is generically globally generated over $Y\setminus (D_1\cup \cdots \cup D_k).$
In particular, if $L+D$ is nef for every torus-invariant prime divisor $D,$ then $T_{\yy}\vert_Y\otimes \oo_X(L)\vert_Y$ is globally generated.
\end{enumerate}
}
\end{proposition}

\begin{proof}\,

(1) Since $T_{\yy/X}\simeq \pi_2^*M_\ee$ by Lemma~\ref{lem:coskun-riedl-prop2.1}(2), this follows from Lemma~\ref{lem:me}.

(2)
We may restrict the short exact sequence \eqref{eqn:vert-tgt-sheaf} to obtain
\begin{equation}
0 \rightarrow T_{\yy/X}\vert_Y \rightarrow T_{\yy}\vert_Y \rightarrow \pi_2^*T_X\vert_Y \rightarrow 0.
\end{equation}
So in order to show that $T_{\yy}\vert_Y\otimes \oo_X(L)\vert_Y$ is generically globally generated outside of $D_1\cup \cdots \cup D_k$, it suffices to show that the following hold:
\begin{enumerate}[label=(\alph*)]
\item $H^1(Y, T_{\yy/X}\vert_Y\otimes \oo_X(L)\vert_Y) \rightarrow H^1(Y, T_{\yy}\vert_Y\otimes \oo_X(L)\vert_Y)$ is injective, or equivalently,\\
$H^0(Y, T_{\yy}\vert_Y\otimes \oo_X(L)\vert_Y) \rightarrow H^0(Y, \pi_2^*T_X\vert_Y\otimes \oo_X(L)\vert_Y)$ is surjective.
\item $T_{\yy/X}\vert_Y\otimes \oo_X(L)\vert_Y$ is globally generated, which follows from part (1).
\item $T_X(L)\vert_Y$ is generically globally generated outside of $D_1\cup \cdots \cup D_k.$
\end{enumerate}

By Lemma~\ref{lem:mult-map-surj}(2), the multiplication map $H^0(X,\ee) \otimes H^0(X,\oo_X(L)) \rightarrow  H^0(X,\ee(L))$ is surjective. 
Since $L$ is ample and $H^1(X,L)=0$, $H^0(X,\ee) \otimes H^0(Y,\oo_X(L)\vert_Y) \rightarrow  H^0(Y,\ee(L)\vert_Y)$ is surjective as well. 
Therefore, the map
\[ 
H^1(Y, T_{\yy/X}\vert_Y\otimes \oo_X(L)\vert_Y) \rightarrow H^0(X,\ee) \otimes H^1(Y,\oo_X(L)\vert_Y)
\]
is injective. 
This map factors through
\[
H^1(Y, T_{\yy/X}\vert_Y\otimes \oo_X(L)\vert_Y) \rightarrow H^1(Y,T_\yy\vert_Y\otimes \oo_X(L)\vert_Y),
\]
so it is also injective, and (a) is proved.

Consider the Euler sequence
\begin{equation*}
0 \rightarrow \oo_X^{\oplus s-n} \rightarrow \bigoplus_{i=1}^s \oo_X(D_i) \rightarrow T_X \rightarrow 0.
\end{equation*}
(c) follows since we assume that $\oo_X(L+D_i)$ is globally generated when $i=k+1,\ldots,s$, and we know that $\oo_X(L+D_i)$ is generically globally generated outside of $D_i$ when $i=1,\ldots,k$.
\end{proof}

\begin{proof}[Proof of Theorem~\ref{thm:sec-3-gen-hyp-gor}]

Let $Y\subset X$ be a very general hypersurface in $\left|N+2nL\right|$, and $f:C\rightarrow Y$ a nonconstant map from a smooth projective curve $C$.
Consider a resolution $\rho\colon X' \to X$ so that $X'$ is a smooth projective toric variety sharing the torus $\TT$ with $X$.
Denote by $N'$ and $L'$ the pullbacks of $N$ and $L$ to $X'$, respectively.
Let $Y'\subset X'$ be the hypersurface that is the closure of $Y\cap \TT$ in $X'$, so it is a very general hypersurface in $\left|N'+2nL'\right|$.
Let $f'\colon C\to Y'$ be the lift of $f$ to the closure of $C\cap \TT$ in $Y'$.
We take $\cc,\yy,B$ as in Construction~\ref{construction} with respect to $X'$ and $\ee':=\oo_{X'}(N'+2nL')$.

Suppose that $f'(C)$ meets the torus $\TT$.
Recall from the setup that we have the following short exact sequence \eqref{eqn:vertical-ses-C} on $\mathcal{C}$, where $\mathcal{K}$ injects into $N_{g/\yy}$ with torsion cokernel (see Lemma~\ref{lem:coskun-riedl-prop2.1}(3)):  
\[
0 \rightarrow T_{\cc/X'} \rightarrow g^*T_{\yy/X'} \rightarrow \mathcal{K} \rightarrow 0.
\]
Since $M_\ee\otimes \oo_X(L)$ is nef by Lemma~\ref{lem:me}, $M_{\ee'}\otimes \oo_{X'}(L')$ is generically nef modulo the toric boundary, and so is $T_{\yy/X'}\vert_{Y'}\otimes \oo_{X'}(L')\vert_{Y'}$ by Lemma~\ref{lem:coskun-riedl-prop2.1}(2).
Therefore,
$N_{f'/Y'}\otimes \oo_{X'}(L')\vert_C$ is nef, and 
\begin{equation*}
\deg N_{f'/Y'}\otimes \oo_{X'}(L')\vert_C=2g(C)-2-K_{X'}\cdot C -(N'+2nL')\cdot C + (n-2)L'\cdot C \geq 0.
\end{equation*}
Since $X$ is Gorenstein, \[K_{X'}+(n+1)L'=\rho^*(K_X+(n+1)L)+F\] for some effective divisor $F$ on $X'$ that is exceptional over $X$. 
Since $K_X+(n+1)L$ is nef by \cite{Payne06}, $K_{X'}+(n+1)L'$ is nef, and we have 
\begin{equation}
\label{eqn:degree-normal-bundle}
2g(C)-2\geq (K_{X'}+(n+1)L'+N')\cdot C +L'\cdot C\geq L'\cdot C=L\cdot C.
\end{equation}
This shows that $N+2nL$ is pseudo hyperbolic modulo the toric boundary.
\end{proof}

The proof of the smooth case is similar to the above.

\begin{proof}[Proof of Theorem~\ref{thm:sec-3-gen-hyp-smooth}]

Let $Y\subset X$ be a very general hypersurface in  $\left|N+2nL\right|$, and let $f:C\rightarrow Y$ be a nonconstant map from a smooth projective curve $C$.
Take $\ee=\oo_X(N+2nL)$ and $\cc,\yy,B$ as in Construction~\ref{construction} with respect to $X$ and $\ee$.

Suppose that $f(C)$ meets the torus $\TT$.
We again have the following short exact sequence \eqref{eqn:vertical-ses-C} on $\mathcal{C}$, where $\mathcal{K}$ injects into $N_{g/\yy}$ with torsion cokernel:  
\[
0 \rightarrow T_{\cc/X} \rightarrow g^*T_{\yy/X} \rightarrow \mathcal{K} \rightarrow 0.
\]
Proposition~\ref{prop:universal-tangent-sheaf-glob-gen}(1) implies that $N_{f/Y}\otimes \oo_X(L)\vert_C$ is generically globally generated over $C\cap g^{-1}(\TT)$. 
Therefore, 
\begin{equation*}
\deg N_{f/Y}\otimes L\vert_C=2g(C)-2-K_X\cdot C -(N+2nL)\cdot C + (n-2)L\cdot C \geq 0.
\end{equation*}
Since $K_X+(n+1)L$ is basepoint-free \cite{Mustata02}, we have 
\begin{equation}
\label{eqn:degree-normal-bundle}
2g(C)-2\geq (K_X+(n+1)L+N)\cdot C +L\cdot C\geq L\cdot C.
\end{equation}
This shows that $N+2nL$ is pseudo hyperbolic modulo the toric boundary.

Now suppose that $f(C)$ does not necessarily meet $\TT$, but let us assume that $\{D_1,\ldots,D_k \}$ is the set of torus-invariant prime divisors such that $L+D_i$ is not nef.
By looking at the short exact sequence \eqref{eqn:standard-ses-C} 
\[
0 \rightarrow T_\cc \rightarrow g^*T_\yy \rightarrow N_{g/\yy} \rightarrow 0
\]
on $\mathcal{C}$, Proposition~\ref{prop:universal-tangent-sheaf-glob-gen}(2) implies that $N_{f/Y}\otimes \oo_X(L)\vert_C$ is generically globally generated over $C$ away from \linebreak $g^{-1}(D_1\cup \cdots \cup D_k)$. 
Therefore, the same calculation as above shows that $N+2nL$ is pseudo hyperbolic modulo $D_1\cup \cdots \cup D_k$.
\end{proof}

\subsection{Hyperbolicity in toric varieties}
\label{subsec:hyperb-toric}

In this subsection, we prove a statement about the algebraic hyperbolicity of very general hypersurfaces in smooth projective toric varieties, by applying Theorem~\ref{thm:sec-3-gen-hyp-smooth} together with an induction argument on the dimension of the variety. 

\begin{theorem}\label{thm:hyperbolicity}
Let $X$ be a smooth projective toric variety of dimension $n$. Let $N$ be a nef Cartier divisor on $X$ and $L$ be an ample Cartier divisor on $X$.
Then, the linear system $|N+2nL|$ is hyperbolic.
\end{theorem} 

\begin{proof}
We proceed by induction on the dimension. 
In the case that $X$ has dimension $2$ the hyperbolicity and pseudo hyperbolicity of curves agree. So, Theorem~\ref{thm:sec-3-gen-hyp-smooth} implies the statement in dimension two.

By Theorem~\ref{thm:sec-3-gen-hyp-smooth}, we know that the linear system $|N+2nL|$ is pseudo hyperbolic.
More precisely, it is hyperbolic up to $D_1\cup \dots \cup D_k$, where the $D_i$'s are the torus invariant prime divisors for which $L+D_i$ is not nef. 
Let $D$ be any such invariant prime divisor.
As $X$ is a smooth projective toric variety
the divisor $D$ is a smooth projective toric variety as well.
We have a short exact sequence:
\[
0\rightarrow 
\mathcal{O}_X(N+2nL-D) 
\rightarrow 
\mathcal{O}_X(N+2nL)
\rightarrow 
\mathcal{O}_D(N|_{D}+2nL|_{D})
\rightarrow 
0. 
\]
By Musta\c{t}\u{a}'s vanishing on toric varieties we know that $H^1(X,\mathcal{O}_X(N+2nL-D))=0$ (see~\cite[Theorem 0.1]{Mustata02}).
Therefore, we have a surjective homomorphism:
\[
H^0(X,\mathcal{O}_X(N+2nL))
\rightarrow
H^0(D,\mathcal{O}_D(N|_{D}+2nL|_{D})).
\]
In particular, a very general element of 
$H^0(X,\mathcal{O}_X(N+2nL))$
restricts to a very general element of 
\[
H^0(D,\mathcal{O}_D(N|_{D}+2nL|_{D})).
\]
By induction on the dimension, we know that a general element of the latter linear system is hyperbolic. 
Indeed, the divisor $N|_{D}$ is a nef Cartier divisor, 
$L|_{D}$ is an ample Cartier divisor, and 
$2n>2n-2$.

Let $Y\in |N+2nL|$ be a very general element.
By the previous discussion, we know that $Y$ is hyperbolic up to $D_1\cup \dots \cup D_k$.
On the other hand, for every $i\in \{1,\dots,k\}$, we know that
$Y_i:=Y\cap D_i$ is a hyperbolic algebraic variety.
Let $\epsilon_i>0$ be a constant such that 
\begin{equation}\label{eq:insideDi}
2g(C)-2 \geq \epsilon_i \deg f^*(N|_{D_i}+2nL|_{D_i})
\end{equation}
holds for every morphism $f\colon C\rightarrow D_i$ from a smooth projective curve.
Let $\epsilon_0>0$ be a constant for which
\begin{equation}\label{eq:outsideDi}
2g(C)-2 \geq \epsilon_0 \deg f^*(N+2nL)
\end{equation} 
holds for every morphism $f\colon C\rightarrow X$ from a smooth projective curve $C$ for which $f(C)\not \subset D_1\cup \dots \cup D_k$.
Set $\epsilon:=\min\{\epsilon_0,\epsilon_1,\dots,\epsilon_k\}$.
Then, for every morphism from a smooth projective curve $f\colon C\rightarrow X$, we have 
\[
2g(C)-2 \geq \epsilon \deg f^*(N+2nL).
\]
Indeed, the previous inequality holds by~\eqref{eq:insideDi} and by~\eqref{eq:outsideDi} depending on whether $f(C)$ lies in some $D_i$ or not, respectively.
This finishes the proof of the theorem.
\end{proof}

\begin{corollary}\label{cor:hyper-conj}
Let $X$ be a smooth projective toric variety
of dimension $n$. Let $L$ be an ample Cartier divisor. Then, the linear system $|K_X+(3n+1)L|$ is hyperbolic. Furthermore, if $X$ is not isomorphic to $\pp^n$, then $|K_X+3nL|$ is hyperbolic.
\end{corollary}

\begin{proof}
By Fujita's conjecture for toric varieties~\cite{Payne06}, we know that $K_X+(n+1)L$ is nef. Then, the first statement follows from Theorem~\ref{thm:hyperbolicity}. Further, we know that $K_X+nL$ is already nef provided that $X$ is not isomorphic to the projective space.
\end{proof}

\begin{remark}\label{rem:pn-case}
{\em 
We noted in the introduction that when $n\geq 3$, $X= \pp^n$ and $L$ is an ample line bundle on $X$, it is known that $\left|K_X+dL\right|$ is hyperbolic whenever $d\geq 3n$. In fact, when $n\geq 5$, the sharp bound is $d\geq 3n-1$ \cite{Ein88, Voi96, Pac04, Yeo24}.
Here in the projective space case, our proof only applies when $d\geq 3n+1$, essentially because we need the linear system to be of the form $N+2nL$ for a nef divisor $N$.
When $X$ is a Gorenstein projective toric variety that is not isomorphic to $\pp^n$, we utilize the fact that $K_X+nL$ is nef.
}
\end{remark}

\section{The hyperbolicity conjecture for 
Gorenstein toric 3-folds}

In this subsection, we prove the hyperbolicity conjecture for Gorenstein toric threefolds. 
First, we prove the following lemma regarding linear systems on toric surfaces.

\begin{lemma}\label{lem:toric-surf-divisor}
Let $S$ be a projective toric surface.
Let $N$ be a nef Cartier divisor on $S$ and let $L_1,\dots,L_k$ be ample Cartier divisors on $S$ with $k\geq 5$.
Then, a general element of the linear system $|N+L_1+\dots+L_k|$ is a smooth projective curve of genus at least $2$.
\end{lemma}

\begin{proof}
The statement is clear for $\pp^2$ by the degree-genus formula. So, from now on we assume that $S$ is not isomorphic to $\pp^2$.
Note that the linear system $|N+L_1+\dots+L_k|$ is base point free so its general member
is a normal projective curve, hence smooth.
Let $C$ be such a general element.
Observe that $C$ is contained in the smooth locus of $S$, so we have 
\[
(K_S+C)\cdot C= 2g(C)-2.
\]
Therefore, it suffices to show that the product $(K_S+C)\cdot C$ is positive.
On the other hand, $C$ is an ample curve, so it intersects positively
all the prime torus invariant curves of $S$.
Let $T_1,\dots,T_\ell$ be the torus invariant curves of $S$.
First, we prove the following claim.\\ 

\noindent \textit{Claim:} There exists an effective torus invariant divisor $D\sim_\qq N+L_1+\dots+L_k$ for which $D=\sum_{i=1}^\ell a_iT_i$ and each $a_i\geq 1$. 

\begin{proof}[Proof of the Claim]
By~\cite[Theorem 6.2.12]{CLS11}, we know that $N$ is base point free so we can write 
$N\sim D_1=\sum_{i=1}^\ell b_iT_i$ where each $b_i\geq 0$.
Similarly, $L_1$ is base point free so we can write 
$L_1\sim D_2=\sum_{i=1}^\ell c_iT_i$ where each $c_i\geq 0$.
As the divisor $D_2$ is an ample effective Cartier divisor, it is a connected variety.
Therefore, there are at most two coefficients, let's say $c_1$ and $c_2$ that could be zero, 
while $c_i\geq 1$ for every $i\geq 3$.
By~\cite[Fujino's Theorem +]{Payne06}, we know that both $\qq$-Cartier divisors 
\begin{equation}\label{eq:nef}
L_2+L_3-T_1 \text{ and } L_4+L_5-T_2
\end{equation}
are nef divisors.
Here, we are using the fact that $S$ is not isomorphic to $\pp^2$.
In particular, by~\cite[Theorem 6.2.12]{CLS11} some multiple of either divisor in~\eqref{eq:nef} is base point free.
Therefore, we can write 
\begin{equation}\label{div3}
L_2+L_3 \sim_\qq D_3=\sum_{i=1}^\ell d_i T_i
\end{equation} 
with $d_i\geq 0$ and $d_1\geq 1$,
and 
\begin{equation}\label{div4}
L_4+L_5 \sim_\qq D_4=\sum_{i=1}^\ell e_i T_i
\end{equation}
with $e_i\geq 0$ and $e_2\geq 1$.
For each $L_k$ with $k\geq 6$, we choose $0\leq D_{k-1}\sim L_k$ where $D_{k-1}$ is an effective torus invariant divisor. 
Therefore, it suffices to define $D:=\sum_{i=1}^{\ell-1} D_i$, where the $D_i$'s where constructed as above.
The coefficients of $D$ along $T_1$ and $T_2$ are at least one because such property holds for $D_3$ and $D_4$, respectively.
The coefficients of $D$ along $T_i$, with $i\geq 3$, are at least one, because such property holds for $D_1$. 
Thus, all the coefficients of $D$ are at least one.
\end{proof}

Note that 
\[
K_S+C \sim_\qq K_S+D \sim_\qq K_S + \sum_{i=1}^{\ell} a_i T_i  =
\sum_{i=1}^\ell (a_i-1)T_i. 
\]
If any $a_i>1$, then $T_i\cdot C >0$ and so $(K_S+C)\cdot C>0$ finishing the proof.
If every $a_i=1$, then $-K_S\sim_\qq N+L_1+\dots+L_k$ and so every $K_S$-negative curve
satisfies $-(K_S\cdot C)\geq 5$. This contradicts the cone theorem (see, e.g.~\cite[Theorem 0.1]{Fuj03})
\end{proof}

\begin{proof}[Proof of Theorem~\ref{thm:Fujita-hyperbolicity-toric-3-folds}]
Let $X$ be a Gorenstein toric $3$-fold. 
Let $L$ be an ample Cartier divisor on $X$.
By Theorem~\ref{thm:sec-3-gen-hyp-gor} and Remark~\ref{rem:pn-case}, we know that $|K_X+9L|$ satisfies the hyperbolicity condition for all curves that are not contained in the toric boundary. 
Let $D$ be a torus invariant prime divisor of $X$.
Let $\psi\colon X'\rightarrow X$ be a small $\qq$-factorialization of $X$.
Let $D'$ be the strict transform of $D$ in $X'$.
Let $E$ be a torus invariant prime divisor on $X'$ which is ample over $X$.
Then, the divisor
\[
(\phi^*(K_X+9L)-D') + D' +\epsilon E,
\]
is an ample $\qq$-divisor on $X'$ for $\epsilon$ small enough.
Hence, by~\cite[Theorem 0.1]{Mustata02}, we know that 
\[
H^1(X',\mathcal{O}_{X'}(\phi^*(K_X+9L)-D'))=0,
\]
and so the restriction homomorphism 
\[
H^0(X',\mathcal{O}_{X'}(\phi^*(K_X+9L)))
\rightarrow 
H^0(D', \mathcal{O}_{D'}((\phi^*(K_X+9L))|_{D'}))
\]
is surjective. For the previous surjection, we are using the fact that $D'$ is a $\qq$-Cartier divisor (see, e.g.~\cite[Lemma 2.42]{Bir19}).
Note that we have a commutative diagram induced by pushing-forward sections:
\[
\xymatrix{
H^0(X',\mathcal{O}_{X'}(\phi^*(K_X+9L)))\ar[r]\ar[d] &  
H^0(D', \mathcal{O}_{D'}((\phi^*(K_X+9L))|_{D'}))\ar[d] \\
H^0(X,\mathcal{O}_X(K_X+9L))\ar[r]
&
H^0(D,\mathcal{O}_D((K_X+9L)|_D)).
}
\]
The vertical arrows of the previous commutative diagram are isomorphisms.
Therefore, we conclude that the homomorphism 
\[
H^0(X,\mathcal{O}_X(K_X+9L))
\rightarrow 
H^0(D,\mathcal{O}_D((K_X+9L)|_D))
\]
is surjective.
Hence, a very general element of $H^0(X,\mathcal{O}_X(K_X+9L))$ restricts to a very general element of $H^0(D,\mathcal{O}_D((K_X+9L)|_{D}))$.
By Fujita's conjecture for toric varieties (see, e.g.~\cite[Fujino's Theorem]{Payne06}), we know that the linear system $|K_X +4L|$ is base point free. 
Therefore, by Lemma~\ref{lem:toric-surf-divisor}, we know that a general element of 
\[
|(K_X+4L)|_{D} + 5(L|_D)|
\]
has genus at least $2$.
Thus, we conclude that the linear system $|K_X+9L|$ is hyperbolic.
\end{proof}

\section{Examples and questions}
\label{sec:examples}

In this section, we collect some examples and further questions.
First, we provide some examples showing that Conjecture~\ref{conj:main} holds for products of projective spaces and products of Grassmannians.

\begin{example}\label{ex:prod-project}
{\em 
Let $Y$ be a very general hypersurface of bidegree $(d_1,d_2)$ in $X=\pp^{n_1}\times \pp^{n_2}$, and denote the dimension of $X$ by $n=n_1+n_2.$ 
Recall that $K_X = (-n_1-1)H_1 + (-n_2-1)H_2$ where $H_i$ denote the pullback of the hyperplane divisors. 
It follows from \cite{Yeo24} that $Y$ is algebraically hyperbolic if $d_1\geq n+n_1-2$ and $d_2\geq n+n_2-2$ when $n\geq 5$, and if $d_1\geq n+n_1-1$ and $d_2\geq n+n_2-1$ when $n\geq 3$. 
(See also \cite{CoskunRiedl23, HaaseIlten21} and Example 1.4 in \cite{Ikeda09}.)
Since $$K_X + (3n+1)(H_1+H_2) = (2n+n_2)H_1+(2n+n_1)H_2,$$
and we always have $2n+n_2\geq n+n_1-1$ and $2n+n_1\geq n+n_2-1$,
Conjecture~\ref{conj:main} holds in the case of $X=\pp^{n_1}\times \pp^{n_2}$.
} 
\end{example} 

\begin{example}\label{ex:prod-grass}
{\em
Results of Mioranci in \cite[Theorem 1.2]{Mioranci23} confirm Conjecture~\ref{conj:main} for some other examples of homogeneous varieties, namely Grassmannians, products of Grassmannians, orthogonal Grassmannians, symplectic Grassmannians, and flag varieties.

Consider the product $X=G(k_1,n_1)\times G(k_2,n_2)$ of two Grassmannians embedded via the Pl\"ucker embedding in $\pp^{N_1}\times \pp^{N_2}.$
Let $H_i$ be the pullback of the hyperplane divisor under each projection $X\rightarrow \pp^{N_i}$, and denote $\dim X =n.$
Let $Y$ be a very general hypersurface of bidegree $(d_1,d_2)$ with respect to $H_1,H_2.$
Mioranci proved that $Y$ is algebraically hyperbolic if $d_i\geq n + n_i-2$ for each $i=1,2$ when $n\geq 4.$
Since $K_X = -n_1H_1-n_2H_2$ and
\[ K_X+(3n+1)(H_1+H_2) = (2n+n_2+1)H_1+(2n+n_1+1)H_2, \]
we see that Conjecture~\ref{conj:main} holds for this example.
}
\end{example}

\begin{example}
{\em
Let $X$ be a very general complete intersection in $\pp^n$ of type $(d_1,\ldots,d_{k-1})$.
For example, we may take $\sum_{i=1}^{k-1}d_i\geq n+2$ so that $X$ is a variety of general type.
Consider the hypersurface $Y\subset X$ obtained by intersecting $X$ with a very general hypersurface in $\pp^n$ of degree $d_k$ so that $Y$ is a very general complete intersection of type $(d_1,\ldots,d_{k-1},d_k)$ in $\pp^n.$
Then it follows from \cite{Ein88,Ein91} that $Y$ is algebraically hyperbolic if $d_1+\cdots+d_{k-1}+d_k\geq 2n-k+1.$
By adjunction,  $K_X=\left(\sum_{i=1}^{k-1}d_i-n-1\right)H$ where $H$ denotes the restriction of the hyperplane divisor to $X$.
Since
\[ K_X+(3\dim X+1)H=K_X+(3(n-(k-1))+1)H = \left(\sum_{i=1}^{k-1}d_i+2n-3(k-1)\right)H, \]
and we always have $\sum_{i=1}^{k-1}d_i+2n-3(k-1)\geq 2n-k+1-\sum_{i=1}^{k-1}d_i,$ Conjecture~\ref{conj:main} holds for very general complete intersections $X\subset \pp^n.$
}
\end{example}

Now, we turn to give examples of ample Cartier divisors $L$ on smooth toric Fano varieties $X$
for which $L+D$ is not nef for some prime torus invariant divisor $D$.

\begin{example}\label{ex:xld}
{\em 
In this example, we classify triples $(X,L,D)$ consisting of 
\begin{itemize}
\item a smooth toric Fano $3$-folds, 
\item $L$ is an ample Cartier divisor on $X$, and 
\item $D$ is a prime torus invariant divisor of $X$,
\end{itemize} 
such that the divisor $L+D$ is not nef. These triples provide examples of varieties in which Theorem~\ref{thm:sec-3-gen-hyp-smooth} does not immediately imply the smooth case of Theorem~\ref{thm:Fujita-hyperbolicity-toric-3-folds}.

First, we argue that $D\simeq \pp^2$, $\mathcal{O}_X(L)|_D \simeq \mathcal{O}_D(1)$, and 
$\mathcal{O}_X(-K_X)|_D \simeq \mathcal{O}_D(1)$.
Note that any curve $C$ for which $(L+D)\cdot C <0$ must be contained in the effective divisor $D$.
Observe that $D$ is a smooth toric surface.
By adjunction, we have:
\begin{equation}\label{eq:isoms}
\mathcal{O}_D((L+D)|_D) \simeq
\mathcal{O}_D((K_X+D+L-K_X)|_D) \simeq 
\omega_D \otimes \mathcal{O}_D(L|_D) \otimes 
\mathcal{O}_D(-K_X|_D).
\end{equation}
If $D$ is not isomorphic to $\pp^2$, then the cone of effective curves of $D$ is generated by curves $C$ for which $K_D\cdot C \geq -2$.
On the other hand, note that 
$L|_D$ and $-K_X|_D$ are ample Cartier divisors on $D$, so $L|_D\cdot C \geq 1$ and $-K_X|_D \cdot C\geq 1$ for every curve $C\subset D$.
Therefore, by the sequence of isomorphisms~\eqref{eq:isoms}, we conclude that $L+D$ is nef unless $D\simeq \pp^2$.
If $D\simeq \pp^2$, then $L+D$ is not nef along $D$ if and only if $-K_X|_D \sim L|_D \sim \ell$, where $\ell$ is the class of a line in $\pp^2$.
Fano manifolds containing a projective space as a divisor are well studied. 
For instance, Bonavero~\cite[Theorem 2]{Bon02} has classified smooth toric Fano $n$-folds containing $\pp^{n-1}$ as a divisor.
Using Bonavero's result and the fact that $-K_X|_{D}\sim \ell$, we conclude that either:
\begin{enumerate}
\item $X\simeq \pp(\mathcal{O}_{\pp^2}\oplus \mathcal{O}_{\pp^2}(2))$ and $D$ is a section, or 
\item $X\simeq {\rm Bl}_C({\rm Bl}_p \pp^2)$ where $p$ is a point in $\pp^2$, $C$ is a projective line in the exceptional divisor $E$ of the first blow-up, and $D$ is the strict transform of $E$ in $X$.
\end{enumerate} 
First, assume that $X\simeq \pp(\mathcal{O}_{\pp^2}\oplus \mathcal{O}_{\pp^2}(2))$. Let $\pi\colon X\rightarrow \pp^2$ be the projection and $D$ be the section with $D|_D\simeq \mathcal{O}_D(-2)$. Then, the divisor $L$ is an ample Cartier divisor and $L|_{D}\sim \ell$ if and only if we can write 
\[
L \sim \alpha D  + (2\alpha +1)\pi^*\ell 
\]
where $\alpha \geq 1$. 

Now, assume that $X\simeq {\rm Bl}_C({\rm Bl}_p \pp^2)$. 
Let $E$ be the exceptional divisor of $X\rightarrow {\rm Bl}_p(\pp^2)$, let $D$ be the strict transform on $X$ of the exceptional divisor of ${\rm Bl}_p(\pp^2)\rightarrow \pp^2$, and let $\Gamma$ be the preimage of a general line with respect to the fibration $X\rightarrow \pp^2$. Therefore, assuming that $L$ is an ample Cartier divisor and $L|_D\sim \ell$, we can write 
\[
L\sim \alpha D + \beta \Gamma + (2\alpha - \beta +1)E, 
\]
where $\alpha \geq 1, \beta \geq 1$, and 
$2\alpha \geq \beta > \frac{4}{3}\alpha + \frac{2}{3}$.

Therefore, there are only two isomorphism classes for the pairs $(X,D)$. The possible linear equivalence classes for $L$ are parametrized by a sequence in the former case and by the integral points in a $2$-dimensional polyhedron in the latter. 
For any $L$ as above, the homomorphism 
\[
H^0(X,\mathcal{O}_X(K_X+9L))\rightarrow 
H^0(D,\mathcal{O}_D((K_X+9L)|_D))\simeq 
H^0(\pp^2,\mathcal{O}_{\pp^2}(8))
\]
is surjective.
Therefore, a general element of $H^0(X,\mathcal{O}_X(K_X+9L))$ cuts out a general smooth octic curve in $D\simeq \pp^2$.
}
\end{example}

\subsection{Questions}
Finally, we record some questions for further research. 
In Example~\ref{ex:xld}, we describe smooth toric Fano $3$-folds $X$ and ample Cartier divisors $L$ for which there exists a prime torus invariant divisor $D$ with $L+D$ not nef. 
In such an example, we show that $D$ must be isomorphic to $\pp^2$. In higher dimensions, it is possible to produce examples in which $D$ is not isomorphic to projective space.
This leads to the following question.

\begin{question}
What is the geometry of the torus invariant prime divisors $D$ for which $L+D$ is not nef in a polarized toric Fano variety $(X,L)$?
\end{question}

Fujita's conjecture is generally expected to hold for linear systems of the form $\left|K_X+L\right|$ whenever $L$ is a sufficiently positive line bundle.
In~\cite{Mur24}, Murayama considers a variation of Fujita's conjecture for linear systems of the form 
$\left|K_X+aK_X+bL\right|$, proving several results for surfaces. 
In the toric case, the results in~\cite{Fuj03,Payne06} show that  the linear system $\left|K_X +L_1+\dots+L_{n+1}\right|$ is basepoint-free for possibly distinct ample Cartier divisors $L_1,\dots,L_{n+1}$. 
The effective basepoint-freeness of adjoint linear systems has also been studied through the concept of convex Fujita numbers in, for example, \cite{CKMS23, CKMS24b, CKMS24a}.
This leads to the following question.

\begin{question}
To what extent does Conjecture~\ref{conj:main} hold for systems of the form $|K_X+L_1+\dots+L_{3n+1}|$ where the $L_i$'s are ample Cartier divisors? 
\end{question}

\bibliographystyle{habbvr}
\bibliography{references}

\end{document}